\documentclass[12pt, reqno]{amsart}
\usepackage{amscd,amsmath,amsthm,amssymb,graphics}
\usepackage{amsfonts,amssymb,amscd,amsmath,enumitem,verbatim}
\usepackage[a4paper,top=3cm,left=3cm,right=3cm]{geometry}
\theoremstyle{plain}
\usepackage{color}
\usepackage{hyperref}
\newtheorem{Theorem}{Theorem}
\newtheorem{Lemma}[Theorem]{Lemma}

\newtheorem{Proposition}[Theorem]{Proposition}

\theoremstyle{definition}
\newtheorem{Definition}[Theorem]{Definition}
\newtheorem{Remark}[Theorem]{Remark}

\newtheorem{Algorithm}[Theorem]{Algorithm}

\usepackage{graphicx}
\usepackage{xfrac} 
\usepackage{faktor}

\title{Novel performant primality test on a Pell's cubic}
\author[L.~Di Domenico]{Luca Di Domenico} 
\address{Dipartimento di Matematica, Università di Trento, Via Sommarive 14, 38123, Povo (TN), Italy}
\email{luca.didomenico99@gmail.com}

\author[N.~Murru]{Nadir Murru} 
\address{Dipartimento di Matematica, Università di Trento, Via Sommarive 14, 38123, Povo (TN), Italy}
\email{nadir.murru@unitn.it}
\date{}

\keywords{primality test; Pell's cubic; Pell's equation; integer sequences;}

\subjclass{11A51, 11Y11}

\begin{document}

\begin{abstract}
    Primality testing is an especially useful topic for public-key cryptography.
    In this paper, a novel primality test algorithm based on the Pell's cubic will be introduced, and its necessary primality conditions will be proved using three integer sequences connected to operations applied in the projectivization of the Pell's cubic.
    The number of operations involved in the test grows linearly with respect to the bit-length $\log_2(n)$ of the input integer $n$.
    The algorithm is deterministic for integers less than $2^{32}$.
\end{abstract}

\maketitle

\section{Introduction}
Primality testing algorithms are procedures which classify a given odd integer $n$ as prime or composite.
If $n$ is classified composite, then $n$ is surely composite, whereas the test is allowed to wrongly classify composites as primes (these misclassified composites are usually called \textit{pseudoprimes}).
Some of these procedures are fast and return few pseudoprimes.

Primality testing is an interesting topic per se, but it can be especially useful in public-key cryptography.
Currently, the most famous and used primality testing algorithms are the Rabin--Miller test \cite{RM} and Baille--PSW \cite{Baillie}, which are mainly based on Fermat and Lucas tests.
Another famous algorithm is the notorious AKS test \cite{AKS} which was proved to be a deterministic and polynomial--time algorithm, but
it is somewhat slower in practice and the first two tests are always preferred in applications. 

The study of existing primality tests as well as the development of new algorithms is still a very interesting and active research field. For instance, in \cite{EP24, EW24}, the authors study the probability to fail of the strong Lucas test, focusing on the average case and in \cite{BFW} the authors provide an improvement of the Baillie-PSW test. Regarding the development of new tests, in \cite{bad}, the study of bad witnesses for the Galois test and the Rabin--Miller test, inspired the authors for deducing a new primality test which is the product of a multiple-rounds Miller-Rabin test by a Galois test. In \cite{Z}, the authors developed an automatized way for searching many Lucas and Perrin like primality tests. Moreover, it is worthy of interest the study of primality tests for integers of special forms, see, e.g., \cite{Ramzy}.

In this paper, a novel primality test will be introduced, showing some interest from both a theoretical and practical point of view. Indeed, thanks to a tested implementation, it was seen that the test is a primality criterion under $2^{32}$, that is, the test returns zero odd pseudoprimes bigger than $4$ and lower than $2^{32}$. It also has the characteristic of having a number of operation which grows linearly with respect to the length in base $2$ of the input integer, which is rather remarkable. Moreover, from a theoretical point of view, it is based on some properties and construction involving the Pell's cubic and third--order linear recurrent sequences, which have been exploited for the first time in this context.
There are some works which characterise pseudoprimes for primality tests based on sequences of degree $3$ or higher, but none of them introduce an analogous of the test that we will present in this paper. For instance, in
\cite{PseudoprimesForSomeThirdOrderSequences}, the author only talks about third order linear recurrence sequences satisfying the relation $A_{n + 3} = rA_{n + 2} - sA_{n + 1}+A_{n}$ for any integer $n$ and initial conditions $A_{-1} = s, A_{0} = 3$ and $A_{1} = r$.
In \cite{PseudoprimesHighOrderSequences}, the author studies suitable higher-order analogues for the classical Lucas sequences and sequences which arise from resolvent polynomials.
Recently, the Pell conic has been also exploited for the definition of a primality test \cite{PrimalityTestAndPellEquation}. 

The paper is structured as follows.
Section 2.1 introduces the construction of the Pell's cubic, of its group structure and of its norm before characterizing its projectivization.
Section 2.2 highlights an isomorphism between the projectivization of a Pell's cubic and the unitary elements under the Pell's cubic norm in a polynomial ring.
Section 2.3 defines three third-order linear recurring sequences and correlates them with an exponentiation of an element inside the projectivization of the Pell's cubic.
Section 2.4 writes explicitly the Binet's formulae of the three sequences introduced in Section 2.3.

Section 3 is devoted to proving novel necessary conditions of primality, thanks to which our primality test is constructed.

Section 4.1 is dedicated to explaining the rationale regarding the choice of a parameter which is crucial for defining the primality test.
Section 4.2 delves into the variant of the square-and-multiply operation which the primality test uses.
Section 4.3 defines a novel linear primality test based on the necessary conditions seen in Section 3.

\section{On the Pell's cubic}

\subsection{The projectivization of the Pell's cubic over a finite field}
In the following, unless specified otherwise, $p>3$ will be a prime and $r$ will be a non-zero element in the finite field $\mathbb{F}_p$.
Consider the following polynomial ring:
\[
\mathcal{R}_p:= \faktor{\mathbb{F}_p[t]}{\langle t^3 - r \rangle},
\]

its elements are equivalence classes of the form
\[
[x + yt + zt^2] := \left\{ x + yt + zt^2 + k(t)(t^3-r) \mid k(t) \in \mathbb{F}_p[t] \right\}
\]
and such quotient ring inherits a scalar product (from the scalar product of $\mathbb{F}_p[t]$) together with the following operation (derived from the polynomial product in $\mathbb{F}_p[t]$):
\[
\begin{array}{c}
     \left[ x_1 + y_1 t + z_1 t^2 \right] \cdot_r \left[ x_2 + y_2 t + z_2 t^2 \right] := \\
     \left[ x_1 x_2 + r(y_1z_2 + z_1 y_2) + (x_1 y_2 + y_1 x_2 + r z_1 z_2) t + (x_1 z_2 + y_1 y_2 + z_1 x_2) t^2 \right]
\end{array}
\]

The ring $\mathcal{R}_p$ can be interpreted as a cubic number field in the case when $t^3 - r$ is irreducible and when we consider $\mathbb Q$ instead of $\mathbb F_p$. Thus, we can define a norm to be applied over an element in $\mathcal{R}_p$ as follows:
\[
N_r([x + y t + zt^2]):= x^3 - 3xyzr + y^3 r + z^3 r^2 \in \mathbb{F}_p.
\]

The above norm is obtained by multiplying the element $[x +yt + zt^2] \in \mathcal{R}_p$ with its conjugates $[x + y \omega t + z \omega^2 t^2]$ and $[x+ y \omega^2 t + z \omega t^2]$, where $\omega$ is the primitive cubic root of unity lying in a certain finite extension field of $\mathbb{F}_p$.

At this stage, it is possible to construct a trivial group isomorphism between the unitary elements of $\mathcal{R}_p$ with respect to the norm $N_r$ defined as
\[
    U(\mathcal{R}_p) := \{ [x + yt + zt^2 ] \in \mathcal{R}_p \mid N_r([x + yt + zt^2 ]) = 1 \}
\]
and the \textit{Pell cubic over $\mathbb{F}_p$}
\[
    C_p := \{ (x,y,z) \in (\mathbb{F}_p)^3 \mid x^3 - 3rxyz + ry^3 + r^2 z^3 = 1 \}
\]
Thanks to the isomorphism, the operation $\cdot_r$ of $\mathcal{R}_p$ corresponds to an operation of the Pell's cubic $C_p$ called \textit{generalised Brahmagupta product}:
\[
\begin{array}{c}
     \left( x_1, y_1, z_1 \right) \odot_r \left( x_2, y_2, z_2 \right) := \\
     \left( x_1 x_2 + r(y_1z_2 + z_1 y_2), \ x_1 y_2 + y_1 x_2 + r z_1 z_2, \ x_1 z_2 + y_1 y_2 + z_1 x_2 \right)
\end{array}
\]
for all $(x_1, y_1, z_1), (x_2, y_2, z_2) \in C_p$. Clearly,
$(C_p, \odot_r)$ is a commutative group with identity $(1,0,0)$ where the inverse of a certain element $(x,y,z)$ is given by the product of its conjugates
\[
    (x, y \omega, z \omega^2) \odot_r (x, y \omega^2, z \omega t^2) = (x^2 - ryz, rz^2 - xy, y^2 - xz).
\]

We also define the following set
\[
    I(\mathcal{R}_p) := \{ [x + yt + zt^2] \in \mathcal{R}_p \mid N_r([x + yt + zt^2]) \neq 0 \}
\]

which will now allow us to introduce a parametrization of the Pell's cubic $C_p$.

\begin{Definition}
Given a prime $p > 3$ and the set $I(\mathcal{R}_p)$, we define the \textit{projectivization of the Pell's cubic over the finite field $\mathbb{F}_p$} as
\[
\mathbb{P}_p := \faktor{I(\mathcal{R}_p)}{\mathbb{F}_p^*}
\]
The elements of $\mathbb{P}_p$ are equivalence classes denoted as follows:
\[
[x:y:z]:= \{\lambda [x + yt + zt^2] \mid \lambda \in \mathbb{F}_p^* \text{ and } [x + yt + zt^2] \in I(\mathcal{R}_p).\}
\]
If $z$ is non-zero in $\mathbb{F}_p$, then $[x + yt + zt^2]$ is equivalent to $[xz^{-1} + yz^{-1}t + t^2]$, and we take $[xz^{-1} : yz^{-1} : 1]$ as the \textit{canonical representative} of such equivalence class in $\mathbb{P}_p$.
In the case that $z$ is zero and $y$ is non-zero in $\mathbb{F}_p$, then we will take $[xy^{-1}:1:0]$ as canonical representative.
Otherwise, if both $y$ and $z$ are zero, then we take $[1:0:0]$ as the canonical representative.
\end{Definition}

Since the generalised Brahmagupta product $\odot_r$ of the Pell's cubic $C_p$ consists of homogeneous polynomials, then the following operation is well-defined over $\mathbb{P}_p$:
\begin{equation}
\label{Eq:ProjectiveProduct}
\begin{array}{c}
     \left[ x_1: y_1: z_1 \right] \otimes_r \left[ x_2 : y_2 : z_2 \right] := \\
     \left[ x_1 x_2 + r(y_1z_2 + z_1 y_2) : x_1 y_2 + y_1 x_2 + r z_1 z_2 : x_1 z_2 + y_1 y_2 + z_1 x_2 \right].
\end{array}
\end{equation}
Furthermore, it holds that the projectivization $\mathbb{P}_p$ together with $\otimes_r$ forms a commutative group with identity $[1:0:0]$ and with the inverse element of $[x:y:z]$ given by $[x^2 - ryz : rz^2 - xy : y^2 - xz]$.
In particular, the form of the inverse implies that
\begin{equation}
    \label{Eq:FormOfInverseElementInPp}
    [x:1:0] \otimes_r [x^2 : - x : 1 ] = [1:0:0] \in \mathbb{P}_p.
\end{equation}

In \cite{DuttoMurru2023}, the authors provided a fully characterisation of the quotient $\mathcal{R}_p$ and the projectivization $\mathbb{P}_p$.
We shall now quickly summarize the two cases we will use later.

First, consider a prime $p \equiv 1 \pmod 3$. It is always possible to find some non-cube elements in $\mathbb{F}_p$.
If we set $r$ as the \textit{smallest non-cube} of $\mathbb{F}_p$ in the set of canonical representatives modulo $p$, then $\mathcal{R}_p$ is a field together with $\cdot_r$ and with the operation inherited from the polynomial addition of $\mathbb{F}_p[t]$.
Also, the order of the projectivization group $(\mathbb{P}_p, \otimes_{r})$ is exactly $p^2 + p + 1$, and we have the following trivial characterization:
\begin{equation}
    \label{Eq:PointsPp1mod3}
    \mathbb{P}_p = \{[x:y:1], \ [x:1:0], \ [1:0:0] \mid x,y \in \mathbb{F}_p\}.
\end{equation}

Now consider an odd prime $p \equiv 2 \pmod 3$. Each element of $\mathbb{F}_p$ is a cube, and it holds that $\mathcal{R}_p$ is simply a ring under any choice of non-zero element $r \in \mathbb{F}_p$.
Moreover, the order of the projectivization group is $p^2 - 1$.
The point characterization in this case is more involved, and it also uses the only cubic root $s$ of the given $r$ in $\mathbb{F}_p$, where $r$ is the parameter appearing in $(\mathbb{P}_p, \otimes_r)$:
\begin{equation}
\label{Eq:PointsPp2mod3}
\begin{array}{rl}
    \mathbb{P}_p = & \{ [x:y:1],\ [x:1:0], \ [1:0:0] \mid x,y \in \mathbb{F}_p\} \\
     & \setminus \{ [-(m+s)s:m:1], \ [s^2 : s: 1], \ [-s:1:0] \mid m \in \mathbb{F}_p \}.
\end{array}
\end{equation}

From this point onward, if not specified otherwise, we will set the parameter $r \in \mathbb{F}_p$ appearing in the operations $\cdot_r, \odot_r$ and $\otimes_r$ as follows:
\begin{itemize}
    \item given a prime $p \equiv 1 \pmod 3$, then $r$ is the smallest non-cube in $\mathbb{F}_p$;
    \item otherwise, given an odd prime $p \equiv 2 \pmod 3$, we will set $r = 2 \in \mathbb{F}_p$.
\end{itemize}

The following Lemma comes directly from Equations \eqref{Eq:PointsPp1mod3} and \eqref{Eq:PointsPp2mod3} and from the fact that the cubic root of $2$ in $\mathbb{F}_p$ is clearly not $-1 \in \mathbb{F}_p$ for any odd prime $p \equiv 2 \pmod 3$.
\begin{Lemma}
    \label{Lemma:[1:1:0]LiesInPp}
    For any prime $p>3$ and for $r \not= 0, -1$, the element $[1:1:0]$ lies in $(\mathbb{P}_p, \otimes_r)$.
\end{Lemma}

To be precise, given an odd prime $p \equiv 2 \pmod 3$, the element $[1:1:0]$ lies in $(\mathbb{P}_p, \otimes_r)$ for any $r \in \mathbb{F}_p$. 

\subsection{Isomorphisms regarding the projectivization of the Pell's cubic}

In the next section, we will use some group isomorphisms between the projectivization $(\mathbb{P}_p, \otimes_r)$ and $(U(\mathcal{R}_p), \cdot_r)$.
The latter group is defined using the set of unitary elements of the quotient ring $\mathcal{R}_p$ with respect to the norm $N_r$ and using the restricted operation corresponding to the product $\cdot_r$ of the quotient ring itself.

In \cite{DuttoMurru2023} the authors specified isomorphisms between the Pell's cubic $(C_p, \odot_r)$ and the projectivization $(\mathbb{P}_p, \otimes_r)$, from which we can easily obtain an isomorphism between $(C_p, \odot_r)$ and $(U(\mathcal{R}_p), \cdot_r)$. 

Given a prime $p \equiv 1 \pmod 3$ and the smallest non-cube $r$ in $\mathbb{F}_p$, then the following is a group isomorphism:
\begin{equation}
\label{Eq:phi1}
    \begin{array}{rl}
        \phi_1 : & (\mathbb{P}_p, \otimes_r) \xrightarrow{} (U(\mathcal{R}_p), \cdot_r) \\
         & [x:y:z] \mapsto N([x+yt+zt^2])^{(p-4)/3} [x+yt+zt^2]^{\cdot_r 3}
    \end{array}
\end{equation}
where $N([x+yt+zt^2])$ is indeed the norm defined for any element of $\mathcal{R}_p$.
A priori, sending an element of the projectivization $\mathbb{P}_p$ into a cube of $U(\mathcal{R}_p)$ multiplied by a scalar (i.e. a certain integer power of a norm) makes the explicit writing of the inverse of $\phi_1$ quite tricky.
In the proof of Theorem \ref{Thm:KeyTheorem} will be able to maintain the cubic exponent, which means that we can write the following inverse morphism:
\begin{equation}
\label{Eq:phi1Inv}
    \begin{array}{rccl}
         \phi_1^{-1} : & (U(\mathcal{R}_p), \cdot_r) & \longrightarrow & (\mathbb{P}_p, \otimes_r)  \\
         & [x]^{3} & \mapsto & [1:0:0]  \\
         & [x+yt]^{3} & \mapsto & [x/y:1:0]  \\
         & [x+yt+zt^2]^{3} & \mapsto & [x/z:y/z:1]
    \end{array}
\end{equation}

Given an odd prime $p \equiv 2 \pmod 3$ and a non-zero element $r$ in $\mathbb{F}_p$, there exists the following group isomorphism:
\begin{equation}
\label{Eq:phi2}
    \begin{array}{rl}
        \phi_2 : & (\mathbb{P}_p, \otimes_r) \xrightarrow{} (U(\mathcal{R}_p), \cdot_r) \\
         & [x:y:z] \mapsto N([x+yt+zt^2])^{(p-2)/3} [x+yt+zt^2]
    \end{array}
\end{equation}
where $N([x+yt+zt^2])$ is the norm in $\mathcal{R}_p$ and where the inverse map is:
\begin{equation}
\label{Eq:phi2Inv}
    \begin{array}{rccl}
         \phi_2^{-1} : & (U(\mathcal{R}_p), \cdot_r) & \longrightarrow & (\mathbb{P}_p, \otimes_r)  \\
         & [x] & \mapsto & [1:0:0]  \\
         & [x+yt] & \mapsto & [x/y:1:0]  \\
         & [x+yt+zt^2] & \mapsto & [x/z:y/z:1]
    \end{array}
\end{equation}

\subsection{Some integer sequences}

By considering the product $\otimes_r$ and by considering the projective element (also called \textit{point}) $[1:1:0] \in (\mathbb{P}_p, \otimes_r)$, we can fix the following notation.

\begin{Definition}
    Given a non-negative integer $k$, $[1:1:0]^{k} \in \mathbb{P}_p$ is the \textit{$k$-th power of the point $[1:1:0]$ under $\otimes_r$ in $\mathbb{P}_p$}.
    In particular, if $k=0$, then $[1:1:0]^{0} := [1:0:0]$.
    Otherwise, if $k > 0$, then $[1:1:0]^{k}$ represents the equivalence class in $\mathbb{P}_p$ obtained by multiplying together $k$ copies of $[1:1:0]$ under $\otimes_r$.
\end{Definition}


Our aim now is to define integer sequences which are able to mimick the behaviour of each entry of a given power of the point $[1:1:0]$.
To achieve our goal, we shall use $\otimes_r$ to define the following operation between two any elements $(x_1,y_1,z_1),(x_2,y_2,z_2) \in \mathbb{Z}^3$:
\begin{equation}
\label{Eq:IntegerVectorProduct}
\begin{array}{c}
     \left( x_1, y_1, z_1 \right) \times_r \left( x_2, y_2 , z_2 \right) := \\
     \left( \ x_1 x_2 + r(y_1z_2 + z_1 y_2), \ x_1 y_2 + y_1 x_2 + r z_1 z_2, \ x_1 z_2 + y_1 y_2 + z_1 x_2 \ \right)
\end{array}
\end{equation}

\begin{Definition}
    Given a non-negative integer $k$, the vector $(x_k, y_k, z_k)_{\times_r} \in \mathbb{Z}^3$ denotes the \textit{$k$-th power of the vector $(1,1,0)$ under $\times_r$}, where by convention $(x_0, y_0, z_0)_{\times_r}:= (1,0,0)$, $(x_1, y_1, z_1)_{\times_r}:= (1,1,0)$.
\end{Definition}

The proof of the following Lemma is straightforward.

\begin{Lemma}
    Given a prime $p> 3$ and a non-negative integer $k$, fix the positive integer parameter $r$ either as $2$ when $p$ is congruent to $2$ modulo $3$, or as the smallest positive integer such that $r \,(\bmod \, p) \in \mathbb{F}_p$ is a non-cube otherwise.
    Given $(x_k, y_k, z_k)_{\times_r} \in \mathbb{Z}^3$, then the element $[x_k \,(\bmod \, p): y_k \,(\bmod \, p): z_k \, (\bmod\, p)]$ is equivalent to $[1:1:0]^{k}$ in $\mathbb{P}_p$.
\end{Lemma}

Given the following matrix with integer entries
\[
    M:= \begin{pmatrix}
        1 & 0 & r \\
        1 & 1 & 0 \\
        0 & 1 & 1
    \end{pmatrix} \in \mathbb{Z}^{3 \times 3}
\]
the following equalities hold:
\begin{equation}
    \label{Equation:110timesxyz}
    (1,1,0) \times_r (x,y,z) = (\, x + rz,\, x + y, \, y+z \,)
\end{equation}
\[
    M \begin{pmatrix}
        x \\
        y \\
        z
    \end{pmatrix} =
    \begin{pmatrix}
        x + rz \\
        x + y \\
        y + z
    \end{pmatrix}
\]

from which we obtain the following Lemma.

\begin{Lemma}
    \label{Lemma:MatrixVSintegerVector}
    For any integer $r$ and any non-negative integer $k$, it holds that
    \[
    M^k \begin{pmatrix}
        1\\
        0\\
        0
    \end{pmatrix} = 
    \begin{pmatrix}
        x_k \\
        y_k \\
        z_k
    \end{pmatrix}_{\times_r}
    \]
\end{Lemma}

Thanks to the last two Lemmas, if we carefully choose the parameter $r$, then there is a clear correlation between the $k$-th power of the matrix $M$ and the value of $[1:1:0]^{k}$.
It is a well-known fact that each component of the power of a square matrix with integer entries follows a linear recurrence sequence which recurs with the characteristic polynomial of the matrix itself and whose degree coincides with the number of rows (or columns) of the square matrix.
In our case, we have that the characteristic polynomial of $M$ is
\[
    f_r(t) := \det (tI_3 - M) = (t-1)^3 - r = t^3 -3t^2 +3t - (r+1)
\]

By the Cayley-Hamilton theorem, we get that
\[
    M^{k+3}
    \begin{pmatrix}
        1 \\
        0 \\
        0
    \end{pmatrix} =
    3 M^{k+2}
    \begin{pmatrix}
        1 \\
        0 \\
        0
    \end{pmatrix}
    - 3 M^{k+1}
    \begin{pmatrix}
        1 \\
        0 \\
        0
    \end{pmatrix}
    + (r + 1) M^{k}
    \begin{pmatrix}
        1 \\
        0 \\
        0
    \end{pmatrix}
\]

The above equality, together with Lemma \ref{Lemma:MatrixVSintegerVector}, directly implies that \textit{each entry of the vector $(x_k,y_k,z_k)_{\times_r}$ recurs with characteristic polynomial $f_r(t)$}.
We can now define the following three integer linear recurrence sequences:

\begin{Definition}
    Given a non-zero integer $r$ and the polynomial $f_r(t):= t^3 -3t^2 + 3t - (r+1) \in \mathbb{Z}[t]$, we denote the integer linear recurrence sequence with initial values $s_0, s_1, s_2$ and with characteristic polynomial $f_r(t)$ as $\{ P(s_0,s_1,s_2,r) \}_{i \geq 0}$.
    Then, we define the three following families of sequences (which will be called \textit{Pell-X, Pell-Y} and \textit{Pell-Z} sequences of parameter $r$, respectively):
    \[
    \begin{array}{c}
         \{ X_i(r) \}_{i \geq 0} := \{ P(1,1,1,r) \}_{i \geq 0} \\
         \{ Y_i(r) \}_{i \geq 0} := \{ P(0,1,2,r) \}_{i \geq 0} \\
         \{ Z_i(r) \}_{i \geq 0} := \{ P(0,0,1,r) \}_{i \geq 0}
    \end{array}
    \]
\end{Definition}

Thanks to all of the above consideration, we have the following proposition.

\begin{Proposition}
    \label{Proposition:SequencesVSpowersOf110}\ 
    \begin{itemize}
    \item Given a prime $p>3$ congruent to $1$ modulo $3$, let $r$ be the smallest positive integer such that $r \, (\bmod \, p)$ is the smallest non-cube in the finite field $\mathbb{F}_p$.
    Given $(x_k, y_k, z_k)_{\times r} \in \mathbb Z^3$, then
    \[
        X_k(r) = x_k, \ Y_k(r) = y_k , \ Z_k(r) = z_k
    \]
    \[
        [X_k(r) \, (\bmod \, p) : Y_k(r) \, (\bmod \, p) : Z_k(r) \, (\bmod \, p)] = [1:1:0]^{\otimes_r k} \in \mathbb{P}_p
    \]
    for all $k>0$.
    \item Given a prime $p>3$ congruent to $2$ modulo $3$ and $(x_k, y_k, z_k)_{\times 2}$, then 
    \[
        X_k(2) = x_k, \ Y_k(2) = y_k , \ Z_k(2) = z_k 
    \]
    \[
        [X_k(2) \, (\bmod \, p) : Y_k(2) \,  (\bmod \, p) : Z_k(2) \, (\bmod \, p)] = [1:1:0]^{\otimes_2 k} \in \mathbb{P}_p
    \]
    for all $k>0$.
    \end{itemize}
\end{Proposition}

The above Proposition clarifies how the Pell-X, Pell-Y and Pell-Z sequences mimick the behaviour of the power of the projective point $[1:1:0]$.
This connection between such integer sequences and $\mathbb P_p$ will be exploited to create the novel primality testing algorithm.

\subsection{Binet's formulae for the Pell-X, Pell-Y and Pell-Z sequences}

Let $\{1, \omega, \omega^2\}$ be the complex cubic roots of unity and consider the integer polynomial $t^3 - r$ for a certain positive integer $r$, whose roots are $\{\rho, \omega \rho, \omega^2 \rho\}$, given $\rho$ one of the cubic root of $r$.
Since the characteristic polynomial $f_r(t)$ of our sequences is of the form $(t - 1)^3 - r$, we have that
\(
    \{\rho + 1, \omega\rho +  1, \omega^2 \rho + 1\}
\) are the all and only pairwise distinct roots of $f_r(t)$
Using the initial conditions, it is immediate to show that the Binet's formulae for the Pell-X, Pell-Y and Pell-Z sequences are
\[ X_n(r) = \frac{(\rho + 1)^n + (\omega \rho + 1)^n + (\omega^2 \rho + 1)^n}{3} \]
\[ Y_n(r) = \frac{(\rho + 1)^n \omega + (\omega \rho + 1)^n + (\omega^2 \rho + 1)^n \omega^2}{3\omega \rho} \]
\[ Z_n(r) = \frac{(\rho + 1)^n \omega^2 + (\omega \rho + 1)^n + (\omega^2 \rho + 1)^n \omega}{3\omega^2  \rho^2} \]

Furthermore, the following congruences hold in the case that $p = 3k + 1$:
\begin{equation}
    \label{Eq:PellSequencesCongruences1mod3}
    \begin{cases}
         X_p(r) \equiv X_{p^2}(r) \equiv X_{p^2 + p}(r) \equiv 1 \, (\bmod \, p) , \cr
         Y_p(r) \equiv r^k \, (\bmod \, p), \ Y_{p^2}(r) \equiv r^{2k} \, (\bmod \, p), \ Y_{p^2 + p}(r) \equiv r^k + r^{2k} \, (\bmod \, p), \cr
         Z_p(r) \equiv Z_{p^2}(r) \equiv 0 \, (\bmod \, p) \text{ and } Z_{p^2 + p}(r) \equiv 1 \, (\bmod \, p)\
    \end{cases}
\end{equation}
If $p = 3k +2$, then these congruences hold instead:
\begin{equation}
    \label{Eq:PellSequencesCongruences2mod3}
    X_p(r) \equiv 1 \pmod p, \quad Y_p(r) \equiv 0 \pmod p,\quad Z_p(r) \equiv r^k \pmod p.
\end{equation}

Using all of the above congruences, we immediately get the following result.
\begin{Proposition}
    Given an integer sequence $\{ A_i \}_{i \geq 0}$, define the rank of appearence of a given integer $n$ as the smallest positive integer $m$ such that $n$ divides $A_m$.
    Then, the rank of appearance of a prime bigger than $3$ congruent to $1$ modulo $3$ in the Pell-Z sequence of parameter $r$ cannot be larger than the prime itself.
    Moreover, the rank of appearance of a prime bigger than $3$ congruent to $2$ modulo $3$ in the Pell-Y sequence of parameter $r$ cannot be larger than the prime itself.
\end{Proposition}

More importantly, the numerous congruences we have written just above will allow us to construct a primality test over the projectivization of the Pell's cubic $\mathbb{P}_p$ working directly with integers instead of considering projective points of $\mathbb{P}_p$.

\section{A necessary condition for primality}

Thanks to what we established in the previous Sections, we can now prove this:
\begin{Theorem}
    \label{Thm:KeyTheorem} \
    \begin{enumerate}
    \item Given a prime $p = 3k +1$ and $r$ the smallest non-cube in $\mathbb{F}_p$, then
    \[
    \begin{array}{rl}
        \langle A1. \rangle & [1:1:0]^{p} = [1:r^k:0] \in (\mathbb{P}_p, \otimes_r) \\
         \langle A2. \rangle & [1:1:0]^{p^2} = [1:r^{2k}:0] \in (\mathbb{P}_p, \otimes_r) \\
         \langle A3. \rangle & [1:1:0]^{p^2 + p} =[1:r^k:0] \otimes_r [1:r^{2k}:0] = [1:-1:1] \in (\mathbb{P}_p, \otimes_r)
    \end{array}
    \]
    \item Given a prime $p= 3k + 2$ and $r\not=0, -1$ in $\mathbb{F}_p$, then
    \[
        \begin{array}{rl}
            \langle B. \rangle & [1:1:0]^{\otimes_r p} = [1:0:r^k] \in (\mathbb{P}_p, \otimes_r) 
        \end{array}
    \]
    \end{enumerate}
\end{Theorem}

\begin{proof}
    For proving $\langle A1. \rangle$, consider the isomorphism $\phi_1$ of Equation \eqref{Eq:phi1} between the projectivization $(\mathbb{P}_p, \otimes_r)$ of the Pell's cubic and the multiplicative subgroup $(U(\mathcal{R}_p), \cdot_r)$.
    By the hypothesis on $r$ and by Lemma \ref{Lemma:[1:1:0]LiesInPp}, we know that $[1:1:0] \in \mathbb{P}_p$ and we have
    \[
        \phi_1([1:1:0]^{p}) = \left( N([1 + t ])^{(p-4) / 3} [1 + t]^{\cdot_r 3}\right)^{p}
    \]
    By the definition of the norm in $\mathcal{R}_p$, $N([1 + t])$ is a well-defined non-zero value in $\mathbb{F}_p$.
    Hence, let $M:= N([1 + t ])^{(p-4) / 3} \in \mathbb{F}_p$, since $\phi_1$ is an isomorphism, the rightmost element lies in $U(\mathcal{R}_p)$, which can be rewritten as follows:
    \[
        \left( N([1 + t ])^{(p-4) / 3} [1 + t]^{\cdot_r 3}\right)^{p} = \left( [M + Mt]^{p}\right)^{3} \in U(\mathcal{R}_p).
    \]
    The computation of the $p$-th power of $[M + Mt]$ in $U(\mathcal{R}_p)$ can be done by computing the $p$-th power of $M + Mt$ in the polynomial ring $\mathbb{F}_p[t]$ and by transferring the result in $U(\mathcal{R}_p)$:
    \[
        (M + Mt)^p = M + M t^p \in \mathbb{F}_p[t] \implies [M + Mt]^{\cdot_r p} = [M + M r^k t] \in U(\mathcal{R}_p)
    \]
    By the observations above and by the hypothesis of $r \neq 0$, the multiplicative inverses of $M$ and of $r^k$ are well-defined in $\mathbb{F}_p$.
    By applying the inverse isomorphism $\phi_1^{-1}$ of Equation \eqref{Eq:phi1Inv} on $[M + M r^k t]^{\cdot_r 3}$, it derives that
    \[
        [1:1:0]^{p} = [MM^{-1}r^{-k}:1:0] = [r^{-k}:1:0] = [1:r^k:0],
    \]
    from which we obtain $\langle A1. \rangle$.
    
    For proving $\langle A2. \rangle$, we can observe that
    \[
        [1:1:0]^{p^2} =([1:1:0]^{p})^{p} = [r^{-k}:1:0]^{p}
    \]
    Following analogous steps as the previous ones, first use $\phi_1$ to get:
    \[
        \phi([r^{-k}:1:0]^{p}) = \left( N([r^{-k} + t])^{(p-4) / 3} [r^{-k} + t]^{3}\right)^{p}
    \]
    Second, note that $N([r^{-k} + t])$ is non-zero, hence $C := N([r^{-k} + t])^{(p-4) / 3}$ is a well-defined non-zero element of $\mathbb{F}_p$
    Now, rewrite the rightmost element of the above Equation as:
    \[
        \left( N([r^{-k} + t ])^{(p-4) / 3} [r^{-k} + t]^{ 3}\right)^{p} = \left( [Cr^{-k} + Ct]^{p}\right)^{3} \in U(\mathcal{R}_p)
    \]
    Then, compute $(Cr^{-k} + Ct)^p$ in the polynomial ring $\mathbb{F}_p[t]$ and send the result in $U(\mathcal{R}_p)$:
    \[
        (Cr^{-k} + Ct)^p = Cr^{-k} + Ct^p \in \mathbb{F}_p[t] \Rightarrow [Cr^{-k} + Ct]^{p} = [Cr^{-k} + C r^k t] \in U(\mathcal{R}_p)
    \]
    Lastly, return in $\mathbb{P}_p$ using $\phi_1^{-1}$ on the element $[Cr^{-k} + C r^k t]^{3}$:
    \[
        [r^{-k}:1:0]^{\otimes_r p} = [Cr^{-k} C^{-1}r^{-k}: 1: 0] = [r^{-2k}:1 :0] = [1:r^{2k}:0]
    \]
    By putting everything together, we conclude the proof of $\langle A2. \rangle$.
    
    For $\langle A3. \rangle$, by the hypothesis that $r$ is a non-cube in $\mathbb{F}_p$, we have that $(\mathbb{P}_p, \otimes_r)$ is a cyclic group of order $p^2 + p + 1$.
    This means that the $(p^2 + p + 1)$-power of  $[1:1:0]$ under $\otimes_r$ in $\mathbb{P}_p$ must be equal to the identity element $[1:0:0]$ of $\mathbb{P}_p$.
   We get that $[1: r^{k}: 0] \otimes_r [1: r^{2k}:0]$ is the multiplicative inverse of $[1:1:0]$ in the projectivization, since it holds that:
    \[
        \begin{array}{c}
            [1:1:0] \otimes_r [1:r^{k} :0] \otimes_r [1: r^{2k}:0] =\\
            = [1:1:0] \otimes_r [1: 1:0]^{\otimes_r p} \otimes_r [1:1:0]^{p^2} =\\
            = [1:1:0]^{p^2 + p + 1} = [1:0:0]
        \end{array}
    \]
    Applying Equation \eqref{Eq:FormOfInverseElementInPp} with $x=1$, we get
    \[
        [1: r^k :0] \otimes_r [1: r^{2k}:0] = [1:1:0]^{-1} = [1^2 : -1: 1] = [1:-1:1] \in \mathbb{P}_p.
    \]
    We conclude the the proof of $\langle A.3 \rangle$ with a trivial rewriting of the above equalities.
    
    For the proof of $\langle B. \rangle$, a reasoning very similar to the one used in $\langle A1. \rangle$ can be constructed.
    Indeed, consider an odd prime $p = 3k +2$, there exists the isomorphism $\phi_2$ which sends elements of $(\mathbb{P}_p, \otimes_r)$ into elements of $(U(\mathcal{R}_p), \cdot_r)$ (see Equation \eqref{Eq:phi2}).
    Using this isomorphism, it holds that
    \[
        \phi_2([1:1:0]^{p}) = \left( N([1 + t])^{(p-2) / 3} [1 + t] \right)^{p}
    \]
    By definition of the norm, $N([1 + t])$ is a non-zero element in $\mathbb{F}_p$, hence $D := N([1 + t])^{(p-2) / 3}$ is a well-defined non-zero element of $\mathbb{F}_p$ which has its multiplicative inverse $D^{-1}$ well-defined in $\mathbb{F}_p$.
    Hence, using the scalar product of $\mathcal{R}_p$, we can rewrite the rightmost element of the last equality as
    \[
    \left( N([1 + t])^{(p-2) / 3} [1 + t] \right)^{p} = [D + Dt]^{p} \in U(\mathcal{R}_p)
    \]
    and we also have
    \[
    (D + Dt)^p = D + Dt^p \in \mathbb{F}_p[t] \Rightarrow [D + Dt]^{p} = [D + D r^k t^2] \in U(\mathcal{R}_p).
    \]
    By hypothesis $r$ is non-zero in $\mathbb{F}_p$, hence the element $r^{-k}$ is well-defined in $\mathbb{F}_p$.
    Return in $\mathbb{P}_p$ using $\phi_2^{-1}$ (Equation \eqref{Eq:phi2Inv}) on the element $[D + Dr^k t^2]$, simplify the terms and write the result in a non-canonical form with a scalar multiplication by $r^{k}$ to get
    \[
    [1:1:0]^{\otimes_r p} = [D D^{-1}r^{-k}: 0: 1] = [r^{-k}: 0:1] = [1:0:r^k]
    \]
    The above chain of equalities proves $\langle B. \rangle$.
\end{proof}

\begin{Remark}
\label{Remark:AlternativeProofKeyTheorem}
    An alternative proof of Theorem \ref{Thm:KeyTheorem} can be constructed considering the connection between the Pell integer sequences and the powers of $[1:1:0]$ under $\otimes_r$ in $\mathbb{P}_p$ (seen in Proposition \ref{Proposition:SequencesVSpowersOf110}) and the congruences in Equations \eqref{Eq:PellSequencesCongruences1mod3} and \eqref{Eq:PellSequencesCongruences2mod3}.
   
\end{Remark}

Using the equalities of Theorem \ref{Thm:KeyTheorem} and the congruences for the integer sequences in Equations \eqref{Eq:PellSequencesCongruences1mod3} and \eqref{Eq:PellSequencesCongruences2mod3}, we get the following necessary primality conditions.
\begin{Theorem}
\label{Definition:NecessaryPrimalityConditions} \
\begin{enumerate}
\item Given a prime $p = 3k + 1$ and given the smallest non-cube $r$ in $\mathbb{F}_p$, then these four conditions must hold for the vector $(x_p, y_p, z_p)_{\times_r}$:
    \begin{itemize}
        \item $x_p \equiv 1 \,(\bmod \,p)$;
        \item $y_p \equiv r^k \,(\bmod \,p)$;
        \item $z_p \equiv 0 \,(\bmod \,p)$;
        \item $y_p + y_p^2 \equiv -1 \,(\bmod \,p)$;
    \end{itemize}
\item Given a prime $p = 3k + 2$, then these three conditions must hold for the vector $(x_p, y_p, z_p)_{\times_2}$:
    \begin{itemize}
        \item $x_p \equiv 1 \,(\bmod \,p)$;
        \item $y_p \equiv 0 \,(\bmod \,p)$;
        \item $z_p \equiv 2^k \,(\bmod \,p)$;
    \end{itemize}
    \end{enumerate}
\end{Theorem}
\begin{proof}
Except the necessary condition $y_p + (y_p)^2 \equiv -1 \,(\bmod \,p)$, all the other necessary conditions can be directly derived by considering Proposition \ref{Proposition:SequencesVSpowersOf110} and Equations \eqref{Eq:PellSequencesCongruences1mod3} and \eqref{Eq:PellSequencesCongruences2mod3}.

Let us focus on $y_p + y_p^2 \equiv -1 \,(\bmod \,p)$.
First, we will prove that $y_{p^2 + p} \equiv y_{p} + y_{p}^2 \, (\bmod \, p)$ and then we will prove that $y_{p^2 + p} \equiv -1 \, (\bmod \, p)$.

By the definition of $\times_r$ (see Equation \eqref{Eq:IntegerVectorProduct}) and of the powers of $(1,1,0)$ under $\times_r$, $(x_{p}, y_{p}, z_{p})_{\times_r} \times_r (x_{p^2}, y_{p^2}, z_{p^2})_{\times_r} = (x_{p^2 + p}, y_{p^2 + p}, z_{p^2 + p})_{\times_r}$.
By computing the $y$-component of the left-hand side of this last equality, we derive that $y_{p^2 + p} = x_p y_{p^2} + y_p x_{p^2} + r z_p z_{p^2}$.
Using this last equality, Proposition \ref{Proposition:SequencesVSpowersOf110} and Equation \eqref{Eq:PellSequencesCongruences1mod3}, we get $y_{p^2 + p} \equiv x_p y_{p^2} + y_p x_{p^2} + r z_p z_{p^2} \equiv (1 \cdot y_{p^2}) + (y_p \cdot 1) + 0 \equiv y_p + y_{p^2} \, (\bmod \, p)$.
The congruence of the first part is proved by both the use of the leftmost and rightmost term of the last chain of congruences and by the fact that $y_{p^2} \equiv (y_p)^2 \,(\bmod \,p)$ (indeed, $(r^k)^2 = r^{2k}$, and these two chain of congruences hold thanks to Proposition \ref{Proposition:SequencesVSpowersOf110} and Equation \eqref{Eq:PellSequencesCongruences1mod3}: $r^k \equiv Y_p(r) \equiv y_p \,(\bmod \,p)$ and $(r^k)^2 \equiv r^{2k} \equiv Y_{p^2}(r) \equiv y_{p^2} \,(\bmod \,p)$).

By Proposition \ref{Proposition:SequencesVSpowersOf110} and Theorem \ref{Thm:KeyTheorem}, we have that there exists a non-zero $\lambda \in \mathbb{F}_p$ such that $(x_{p^2 + p}, y_{p^2 + p}, z_{p^2 + p})_{\times_r} \equiv (\lambda, -\lambda, \lambda)  \,(\bmod \,p)$.
Yet, we can use Proposition \ref{Proposition:SequencesVSpowersOf110} and Equation \eqref{Eq:PellSequencesCongruences1mod3} to derive that $x_{p^2 + p} \equiv 1 \,(\bmod \,p)$ and that $z_{p^2 + p} \equiv 1 \,(\bmod \,p)$.
These two congruences force the value of $\lambda$ to be equal to $1$.
By substituting the value of $\lambda$ with $1$ in the vector congruence and by extracting the $y$-component of this vector congruence, we get that $y_{p^2 + p} \equiv -1 \,(\bmod \,p)$.
\end{proof}

\begin{Remark}
    \label{Remark:TwoPointsOfViewAndHiddenScalars}
    Let us highlight the importance of having two points of view which we used to construct the necessary conditions, one regarding the Pell integer sequences and another one regarding the projectivization of the Pell's cubic.
    By using strictly the Binet's formulae of the integer sequences, we can easily prove only the congruence $y_p + y_p^2 \equiv r^k + r^{2k} \,(\bmod \,p)$, whereas the necessary condition $y_p + y_p^2 \equiv -1 \,(\bmod \,p)$ seems quite tough to prove.
    Yet, this necessary condition allows the novel primality test to become a primality criterion under $2^{32}$.
    Also, not having the integer vector operation $\times_r$ directly derived from the product $\otimes_r$ would result in a less generalisable, less readable and arguably slower implementation of the novel primality test.
    On the other hand, if we implement the novel primality test using purely computations of the powers of the projective point $[1:1:0]$ under $\otimes_r$ in $\mathbb{P}_p$, extra care is required while checking the values of the powers of $[1:1:0]$ against the expected values provided by $\langle A1. \rangle$, $\langle A2. \rangle$, $\langle A3. \rangle$ and $\langle B. \rangle$ of Theorem \ref{Thm:KeyTheorem}.
    Indeed, by definition of equivalence classe in $\mathbb{P}_p$, the point $[x : y : z] \in \mathbb{P}_p$ coincides with any point of the form $[\lambda x : \lambda y : \lambda z]$ where $\lambda$ is a non-zero "hidden" scalar in $\mathbb{F}_p$.
    If we impose extremely strict checks on each component of the $p$-th power of $[1:1:0]$ in an implementation without representing both the resulting point of the computation and the expected point in canonical form, we risk either the misclassification of some primes or a bigger amount of pseudoprimes.
    Moreover, even if we represent these points in canonical form by adding computational complexity, the purely-projectivization point of view gives us zero assurance that the comparison of these two points component-wise is justified.
    There may be a prime $p$ which is misclassified by an implementation of the primality test because $p$ passes the extremely strict checks only under some unpredictable non-canonical form representations. 
    The shown necessary conditions present no hidden scalars to worry about and require no transformations into canonical form, since the Pell integer sequences lead to precise modular-congruence checks in the novel primality test.
\end{Remark}

\section{The primality test}

We are close to presenting a working implementation of our primality test.
But first, it is crucial to spend some words on a handful of details.

\subsection{Choosing the parameters}
Let $p$ be an odd prime congruent to $1$ modulo $3$.
As we said earlier, it is always possible to find non-cubes in $\mathbb{F}_p$, and the parameter $r$ is required to be non-cube to actually construct the projectivization of the Pell's cubic without falling into an edge case.
To mimimize execution times, we are very interested in fixing $r$ to be the smallest value possible, so that overall the product $\times_r$ can be as fast as possible.
A very trivial method of finding the smallest non-cube in $\mathbb{F}_p$ is selecting the smallest positive integer $r$ such that the generalised Euler criterion fails, i.e., such that $r^{(p-1)/3}$ reduced modulo $p$ is not $1$.
Thus, a procedure which checks, in order, numbers between $2$ and $p-1$ (note that $1$ is a cube of $1$ in $\mathbb{F}_p$) until it returns the first number $r$ such that $r^{(p-1)/3} \not \equiv 1 \, (\bmod \, p)$ is always bound to return the smallest non-cube in $\mathbb{F}_p$.

\begin{Remark}
    \label{Remark:TweakedMethodToFindNonCube}
    In an actual implementation of the primality test, it is advised to use a \textit{tweaked procedure to find the parameter $r$} given, as input, an odd integer $n>3$ which is congruent to $1$ modulo $3$ (whose primality is unknown).
    To speed the computation up, the tweaked procedure can check, in order, only each prime between $2$ and $997$ before eventually checking any integer between $998$ and $n-1$.
    Indeed, it is quite evident that a product of two cubes is still a cube, which means that checking only prime candidates for $r$ is sufficient.
    The fact that the tweaked procedure checks all integers after $997$ leads to less memory usage in the implementation.
    Moreover, the tweaked procedure must account for a possible edge case where zero non-cubes are found for a composite input $n$.
    Indeed, for a prime input $p$, checking in order each prime below $p-1$ must return a smallest non-cube as an output.
    Yet, there may exist a composite $n$ for which every non-zero $x$ less than $n-1$ satisfies $x^{(p-1)/3} \equiv 1 \, (\bmod \, n)$. If such case happens, the procedure declares the input $n$ as "composite".
\end{Remark}

\begin{Remark}
    An implementation of the tweaked method written in Remark \ref{Remark:TweakedMethodToFindNonCube} using as input, in order, each element of the odd integer sequence $7,13,19,25,\ldots$ (where each number is congruent to $1$ modulo $3$ and less than $2^{26}$) performs as follows: a parameter $r$ is always found; $2$ is returned $94$ percent of the times; $3$ is chosen around $4$ percent of the times; $5$ is chosen around $1$ percent of the times; each prime up to $163$ included are chosen with a negligible percentage (less than $0.5$ percent of the time).
    It appears as if the tweaked procedure will never declare an input "composite", but further studies must be done before claiming anything for sure.
\end{Remark}

On the other hand, let now $p$ be an odd prime congruent to $2$ modulo $3$.
As shown in Section 3, the necessary conditions (Theorem \ref{Definition:NecessaryPrimalityConditions}) regarding a prime falling in this category for the primality test are written considering $r=2$.
From a purely mathematical perspective, $\mathbb{P}_p$ under $\otimes_u$ is a cyclic group of cardinality $p^2 -1$ for any non-zero $u$ in $\mathbb{F}_p$.
Yet, considering the implementation perspective, not only are we bound to avoid $-1$ as a parameter for the product operation in $\mathbb{P}_p$ when we compute the powers of $[1:1:0]$, but it is of our best interest to choose such parameter to be rather small, so that the execution becomes faster using as little values as possible.
Fixing $r=2$ means that the multiplication by $2$ can be efficiently implemented as a bitshift.

\subsection{The Square-and-Multiply operation}

Because of the form of the necessary conditions of Theorem \ref{Definition:NecessaryPrimalityConditions}, we want to find a way to compute the vector $(x_p, y_p, z_p)_{\times_r}$ and to reduce each component modulo $p$. 
One way to compute these quantities consists in using a variant of the so-called square-and-multiply algorithm.
Instead of applying $p$ times the operation $(1,1,0) \times_r (x,y,z)$ as seen in Equation \eqref{Equation:110timesxyz} (where $x,y,z$ are generic integer values), we use the binary representation of $p$ to compute $(x,y,z) \times_r (x,y,z)$ around $\log_2(p)$ times while adjusting the operation along the way.
More precisely, given generic integer values $(x,y,z)$, Equation \eqref{Eq:IntegerVectorProduct} implies that
\begin{equation}
    \label{Equation:Squaringxyz}
    (x,y,z) \times_r (x,y,z) = (x^2 + 2ryz, 2xy + rz^2, 2xz + y^2)
\end{equation}
Hence, the following generic algorithm can be defined.

\begin{Algorithm}
    \label{Definition:SquareAndMultiplyMethod}
    Given an odd integer $n>3$, given a non-zero parameter $r$ and given the positive integer $e$, if $\sum_{i=0}^{k-1}b_i 2^i$ is the $k$-bit long binary representation of the integer $e$, then apply the following steps.
    \begin{itemize}
        \item Let $x = y = 1$ and $z = 0$.
        \item In order, fix the index $i$ as each of the values $k - 2, k-3,\dots,0$.
        For each choice of $i$, repeat the actions inside the following point.
        \item  Using Equation \eqref{Equation:Squaringxyz} as a guide, update $x,y,z$ with the reduction modulo $n$ of the values $x^2 + 2ryz$, $2xy + rz^2$ and $2xz + y^2$, respectively.
        At this stage, if it happens that $b_i = 1$, then use Equation \eqref{Equation:110timesxyz} to update $x,y,z$ again with the reduction modulo $n$ of the values $x + rz$, $x + y$ and $y + z$, respectively.
        \item After cycling through each of the $k-1$ values of the index $i$ and updating the values of $x,y,z$ as specified in the previous points, we get that $(x,y,z) \equiv (x_e,y_e,z_e)_{\times_r} \,(\bmod \, n)$, where the notation of the "vector congruence" is a shorthand to indicate that the congruences modulo $n$ hold component-wise.
    \end{itemize}
\end{Algorithm}

\begin{Remark}
    From an implementation standpoint, computing the powers of the point $(1,1,0)$ one at a time is immensely slower than the method proposed in Algorithm \ref{Definition:SquareAndMultiplyMethod}.
    Note also these two alternative but slower variants of the square-and-multiply algorithm.
    The first one traverses the binary representation of the exponent from $i=0$ to $i = k-1$.
    It computes procedurally each of the $2^i$-th power of $(1, 1, 0)$ at the end of each loop, and it updates $x,y,z$ using the current available power only when the $i$-th bit in the binary of the exponent is $1$.
    But this "reverse" variant is slower than the one seen in Algorithm \ref{Definition:SquareAndMultiplyMethod} because the former uses a slower generic product instead of the nimble Equation \eqref{Equation:110timesxyz}.
    The second variant (see Algorithm 7.6 of \cite{CryptographyTheoryAndPractice}) uses point inverses and the non-adjacent form (NAF) representation of a binary number.
    It ends up being slightly slower and much more complex than the method shown earlier.
\end{Remark}

\begin{Remark}
    We could be tempted to compute the components reduced modulo $n$ of the vector $(x_p, y_p, z_p)_{\times_r}$ using linear order sequences algorithms.
    Yet, after testing, it appears that Algorithm \ref{Definition:SquareAndMultiplyMethod} is faster.
    For example, Algorithm 2 of \cite{ComputingLinRecurrSequences} (which is based on the powers of the $3 \times 3$ matrix associated to a third-order integer sequence) could be used as a base to define a method to compute the $p$-th power of $(1,1,0)$.
    But computing the powers of a matrix using a square-and-multiply algorithm turns out to be slower than the method shown in Algorithm \ref{Definition:SquareAndMultiplyMethod}, mainly because the update steps taken for the power of the matrix have more operations in them with respect to the simple squaring and multiplying by $(1,1,0)$ operations we find in Equations \eqref{Equation:110timesxyz} and \eqref{Equation:Squaringxyz}.
    Instead, \cite{ManipulatingThirdOrderSequences} talks about a multistep method which turns a third order linear sequence into an equivalent one.
    It is likely that the process of having first a manipulation procedure to obtain an equivalent sequence and second computations over this equivalent sequence is slower than the process of applying directly the square-and-multiply method of Algorithm \ref{Definition:SquareAndMultiplyMethod}.
\end{Remark}

\subsection{The novel primality test}

We can now define the following primality test.

\begin{Algorithm}
    \label{Definition:LinearPrimalityTest}
    Let $n > 3$ be an input odd integer whose primality is unknown.
    Assume that it is not divisible by $3$, otherwise $n$ is surely composite.
    Let $k$ be the quotient of the integer division of $n$ by $3$.
    Apply the following steps.
    \begin{itemize}
        \item If $n$ is congruent to $2$ modulo $3$, compute $(x_n,y_n,z_n)_{\times_2}$ using the method in Definition \ref{Definition:SquareAndMultiplyMethod}.
        If it holds that $x_n \equiv 1 \,(\bmod \, n)$, that $y_n \equiv 0 \,(\bmod \, n)$ and that $z_n \equiv 2^k \,(\bmod \, n)$, then $n$ is declared prime. \\
        If one of these three conditions is false, then $n$ is composite.
        \item Otherwise $n \equiv 1 \,(\bmod \, 3)$, so search the parameter $r$ as in Remark \ref{Remark:TweakedMethodToFindNonCube}.\\
        If no parameter is found, then $n$ is surely composite. \\
        If $r$ is found, then compute $(x_n,y_n,z_n)_{\times_r}$ using the method in Definition \ref{Definition:SquareAndMultiplyMethod}. \\
        If it holds that $x_n \equiv 1 \,(\bmod \, n)$, that $y_n \equiv r^k \,(\bmod \, n)$, that $z_n \equiv 0 \,(\bmod \, n)$ and that $-1 \equiv y_n + y_n^2 \,(\bmod \, n)$, then $n$ is declared prime. \\
        If one of these four conditions is false, then $n$ is composite.
    \end{itemize}
\end{Algorithm}

    This primality test performs rather well in practice. Indeed, it is a \textbf{primality criterion below $2^{32}$}, that is, no composites below $2^{32}$ are declared primes by the test.
    This claim has been proved by implementing the test in Python, and by comparing the output of the implemented test with the output of the Sympy 1.12 \textit{isprime(n)} function (whose answer, as stated in the corresponding documentation, is definitive for $n < 2^{64}$). For the implementation of the test, see the public GitHub repository: \url{https://github.com/Luckydd99/Novel_Performant_Primality_Test.git}.

\begin{Remark}
    \label{Remark:FourthConditionIsNecessary}
    If we use the above test without the necessary primality condition $-1 \equiv y + y^2 \,(\bmod \, n)$ proved in Theorem \ref{Definition:NecessaryPrimalityConditions}, the integers $6189121$, $12262321$ and $14469841$ are the pseudoprimes below $2^{25}$.
\end{Remark}

    If we take inspiration from the approach seen in the Bailey-PSW primality test, we can even construct the following \textbf{strong primality test} based on the projectivization of the Pell's cubic, which also runs faster on average:
    \begin{itemize}
        \item Let $n > 2$ be the input odd integer.
        \item If $n$ coincides with one of the odd primes below $1000$, then $n$ is prime.
        \item If one of the odd primes below $1000$ divides $n$, then $n$ is composite.
        \item If $n$ is not a Fermat probable prime in base $2$ (that is, if $n$ is such that $2^{n-1} \not \equiv 1 \,(\bmod \, n)$), then $n$ is composite.
        \item If $n$ is classified composite by Algorithm \ref{Definition:LinearPrimalityTest}, then $n$ is composite.
        \item Otherwise, $n$ is almost certainly prime.
    \end{itemize}
    Under testing, it was seen that each odd pseudoprime smaller than $2^{37}$ for the Fermat test in base 2 was classified as composite by the linear test of Algorithm \ref{Definition:LinearPrimalityTest}.

A possible future research direction for the test presented in Algorithm \ref{Definition:LinearPrimalityTest} is trying to say if it is actually a criterion or not, and in the second case find a characterization for the pseudoprimes and results about their distribution.
One could try to check if there are any regularities or irregularities for composite integers $n$ when analysing the component of the vector $(x_n, y_n, z_n)_{\times_r}$ which should be congruent to the power of the parameter $r$.

\end{document}